\documentclass[a4paper,12pt]{amsart}

\PII{}
\makeatletter
\def\@setcopyright{\@empty}
\makeatother

\newcommand{\arr}[2]{{{#1}_1,\dots,{#1}_{#2}}}
\newcommand{\Si}[1]{(1-#1^2)}
\newcommand{\Co}[1]{\cos^4\frac{#1}2}
\newcommand{\allp}{1\le p\le\infty}
\newcommand{\Px}[1]{P_{#1}^{(2,2)}}
\newcommand{\E}{E_n(f)_{p,\alpha}}
\newcommand{\Epar}[2]{E_{#1}\left(#2\right)_{p,\alpha}}
\newcommand{\T}[3]{\tau_{#1}^{#2}\left(#3\right)}
\newcommand{\hatT}[3]{\hat\tau_{#1}^{#2}\left(#3\right)}
\newcommand{\Dl}[3]{\Delta_{#1}^{#2}\left(#3\right)}
\newcommand{\w}{\hat\omega_r(f,\delta)_{p,\alpha}}
\newcommand{\wpar}[1]{\hat\omega_r\left(#1\right)_{p,\alpha}}
\newcommand{\norm}[1]{\left\|#1\right\|_{p,\alpha}}
\newcommand{\Py}[1]{P_{#1}^{(0,4)}}
\newcommand{\Lp}{L_{p,\alpha}}
\newcommand{\Copar}[2]{\left(\cos \frac{#1}2\right)^{4#2}}
\newcommand{\prn}[1]{\left(#1\right)}
\newcommand{\brc}[1]{\left\{#1\right\}}

\newcommand{\numericset}[1]{\mathbb #1}
\newcommand{\numN}{\numericset N}

\newtheorem{thm}{Theorem}[subsection]
\newtheorem{lmm}{Lemma}[subsection]
\newtheorem{cor}{Corollary}[subsection]

\newcounter{const}
\numberwithin{const}{thm}
\numberwithin{const}{lmm}

\makeatletter
\newcommand{\Cn}[1][]{%
	\refstepcounter{const}C_{\theconst}%
  \@ifnotempty{#1}{%
    \newcounter{#1}\setcounter{#1}{\arabic{const}}}}
\makeatother

\newcommand{\lastC}{C_{\theconst}}
\newcommand{\prevC}[1][1]{%
	{\countdef\n=255
	 \n=\theconst
	 \advance\n by-#1
	 C_{\number\n}}}

\numberwithin{equation}{section}
\numberwithin{thm}{section}

\renewcommand{\theconst}{\arabic{const}}

\newsavebox\boxcst
\newcommand\boxconst[1]{\sbox\boxcst{$\Cn\label{#1}$}}
\newcommand\boxedconst{\usebox{\boxcst}}

\DeclareMathOperator*\esssup{ess\,sup}

\begin{document}

\title[A theorem of coincidence of classes]{%
	Theorem of coincidence of classes
	for a generalised shift operator%
}

\author{Nimete Sh.~Berisha}
\address{N.~Sh.\ Berisha\\
	Faculty of Mathematics and Sciences\\
	University of Prishtina\\
	N\"ena Terez\"e~5\\
	10000 Prishtin\"e\\
	Kosovo%
}
\email{nimete.berisha@gmail.com}

\keywords{Generalised modulus of smoothness,
	asymmetric generalised shift operator,
	theorem of coincidence of classes%
}
\subjclass{Primary 41A35, Secondary 41A50, 42A16.}
\date{}

\begin{abstract}
	In this paper,
	for a generalised shift operator introduced earlier,
	we prove theorem of coincidence of classes of functions
	defined by the order of best approximation
	by algebraical polynomials
	and the generalised Lipschitz classes
	defined by the generalised shift operator.
\end{abstract}

\maketitle

\section{Introduction}

In~\cite{p-berisha:anal-99,p-berisha:fundam-99},
a generalised shift operator was introduced,
by its means the generalised moduli of smoothness
of given order were defined,
and  Jackson's type theorem
was proved for these moduli.

In the present paper,
we prove the theorem of coincidence of classes of functions
defined by the order of best approximation by algebraic polynomials
and the generalised Lipschitz classes
defined by means of the generalised moduli of smoothness.

\section{Definitions}

Denote by~$L_p$, $1\le p<\infty$,
the set of functions~$f$
measurable in sense of Lebesgue
with summable $p$-th power,
by~$L_\infty$ the set of functions~$f$
bounded almost everywhere in $[-1,1]$,
and
\begin{displaymath}
	\|f\|_p=
	\begin{cases}
		\left(\int_{-1}^1|f(x)|^p\,dx\right)^{1/p},
			& \text{for $1\le p<\infty$},\\
		\esssup_{-1\le x\le1}|f(x)|,
			& \text{for $p=\infty$}.
	\end{cases}
\end{displaymath}

Denote by~$\Lp$ the set of functions~$f$
such that
$f(x)\*\Si{x}^\alpha\in L_p$,
and put
\begin{displaymath}
	\norm f=\|f(x)\Si{x}^\alpha\|_p.
\end{displaymath}

Denote by $\E$ the best approximation
of a function $f\in\Lp$
by algebraic polynomials of degree not greater than $n-1$,
in $\Lp$ metrics,
i.e.,
\begin{displaymath}
	\E=\inf_{P_n}\norm{f-P_n},
\end{displaymath}
where~$P_n$ is an algebraic polynomial
of degree not greater than $n-1$.

By~$E(p,\alpha,\lambda)$
we denote the class of functions $f\in\Lp$
satisfying the condition
\begin{displaymath}
	\E\le Cn^{-\lambda},
\end{displaymath}
where $\lambda>0$
and~$C$ is a constant
not depending on~$n$ $(n\in\numN)$.

Define generalised shift operator
$\hatT t{}{f,x}$
by
\begin{displaymath}
	\hatT t{}{f,x}=\frac1 {\pi\Si x\Co t}
		\int_0^\pi B_{\cos t}(x,\cos\varphi,R)
			f(R)\,d\varphi,
\end{displaymath}
where
\begin{align}\label{eq:R-B}
	R 			&=x\cos t-\sqrt{1-x^2}\sin t\cos\varphi, \notag\\
	B_y(x,z,R) 	&=2\Bigl(\sqrt{1-x^2}y+xz\sqrt{1-y^2}\\
				& \quad
		+\sqrt{1-x^2}(1-y)\Si z\Bigr)^2-\Si R. \notag
\end{align}

By means of this generalised shift operator
define the $r$-th generalised difference by
\begin{align*}     
	\Dl t1{f,x} &=\Dl t{}{f,x}=\hatT t{}{f,x}-f(x),\\
	\Dl{\arr t r}r{f,x} &=\Dl{t_r}{}{\Dl{\arr t{r-1}}{r-1}{f,x},x}
		\quad (r=2,3,\dots),
\end{align*}
and for a function $f\in\Lp$,
define the $r$-th generalised modulus of smoothness
as follows
\begin{displaymath}
	\w=\sup_{\substack{|t_j|\le\delta\\j=1,2,\dots,r}}
		\norm{\Dl{\arr t r}r{f,x}}
	\quad(r=1,2,\dotsc).
\end{displaymath}

Consider the class $H(p,\alpha,r,\lambda)$
of functions $f\in\Lp$
satisfying the condition
\begin{displaymath}
	\w\le C\delta^\lambda,
\end{displaymath}
where $\lambda>0$
and~$C$ is a constant
not depending on~$\delta$.

Put $y=\cos t$, $z=\cos\varphi$ in the operator $\hatT t{}{f,x}$,
denote it by~$\T y{}{f,x}$
and rewrite it in the form
\begin{displaymath}
	\T y{}{f,x}=\frac4{\pi\Si x(1+y)^2}
		\int_{-1}^1 B_y(x,z,R)f(R)\frac{dz}{\sqrt{1-z^2}},
\end{displaymath}
where~$R$ and $B_y(x,z,R)$ are defined in~\eqref{eq:R-B}.

Define the $r$-th power of the generalised shift operator
by
\begin{align*}
	\T{y}1{f,x} &=\T y{}{f,x},\\
	\T{\arr y r}r{f,x} &=\T{y_r}{}{\T{\arr t{r-1}}{r-1}{f,x},x}
		\quad (r=2,3,\dots).
\end{align*}

By~$P_\nu^{(\alpha,\beta)}(x)$ $(\nu=0,1,\dotsc)$
we denote the Jacobi polynomials, 
i.e., the algebraic polynomials of degree~$\nu$,
orthogonal with the weight function
$(1-x)^{\alpha}(1+x)^{\beta}$
on the segment~$[-1,1]$,
and normed by the condition
\begin{displaymath}
	P_\nu^{(\alpha,\beta)}(1)=1 \quad(\nu=0,1,\dotsc).
\end{displaymath}

Denote by~$a_n(f)$ the Fourier--Jacobi coefficients
of a function~$f$,
integrable with the weight function $\Si{x}^2$
on the segment~$[-1,1]$,
with respect to the system of Jacobi polynomials
$\brc{\Px n(x)}_{n=0}^\infty$,
i.e.,
\begin{displaymath}
	a_n(f)=\int_{-1}^1f(x)\Px n(x)\Si{x}^2\,dx\quad(n=0,1,\dotsc).
\end{displaymath}

\section{Auxiliary statements}

In order to prove our results
we need the following theorem.

\begin{thm}\label{th:jackson}
	Let the numbers~$p$ and~$\alpha$ be such that $\allp$;
	\begin{alignat*}2
		1/2 			   &<\alpha\le1
		  &\quad &\text{for $p=1$},\\
		1-\frac1{2p} &<\alpha<\frac32-\frac1{2p}
		  &\quad &\text{for $1<p<\infty$},\\
		1 				   &\le\alpha<3/2
		  &\quad &\text{for $p=\infty$}.
	\end{alignat*}
	If $f\in\Lp$,
	then for every natural number~$n$
	\begin{displaymath}
		\Cn\E\le\wpar{f,1/n},
	\end{displaymath}
	where the positive constant~$\lastC$
	does not depend on~$f$ and~$n$.
\end{thm}

Theorem~\ref{th:jackson}
was proved in~\cite{p-berisha:fundam-99}.
It is known as a Jackson's type theorem.

We also need the following lemmas.

\begin{lmm}\label{lm:properties-tau}
	The operator~$\T y{}{f,x}$ has the following properties:
	\begin{enumerate}
	\item\label{it:properties-tau-1}
		it is linear,
	\item\label{it:properties-tau-2}
		$\T1{}{f,x}=f(x)$,
	\item\label{it:properties-tau-3}
		$\T y{}{\Px\nu,x}=\Px\nu(x)\Py\nu(y)
		\quad(\nu=0,1,\dotsc)$,
	\item\label{it:properties-tau-4}
		$\T y{}{1,x}=1$,
	\item\label{it:properties-tau-6}
		$a_n(\T y{}{f,x})=a_n(f)\Py n(y)
			\quad(n=0,1,\dotsc)$.
	\end{enumerate}
\end{lmm}

Lemma~\ref{lm:properties-tau}
was proved in~\cite{p-berisha:anal-99}

\begin{lmm}\label{lm:bound-tau}
	Let the numbers~$p$ and~$\alpha$ be such that $\allp$;
	\begin{alignat*}2
		1/2				&<\alpha\le1
		  &\quad &\text{for $p=1$},\\
		1-\frac1{2p}	&<\alpha<\frac32-\frac1{2p}
		  &\quad &\text{for $1<p<\infty$},\\
		1				&\le\alpha<3/2
		  &\quad &\text{for $p=\infty$}.
	\end{alignat*}
	If $f\in\Lp$,
	then
	\begin{displaymath}
		\norm{\hatT t{}{f,x}}\le\frac C{\Co t}\norm f,
	\end{displaymath}
	where constant~$C$ does not depend on~$f$ and~$t$.
\end{lmm}

Lemma~\ref{lm:bound-tau} was proved in~\cite{p-berisha:anal-99}.

\begin{cor}\label{cr:bound-tau}
	Let be given numbers~$p$, $\alpha$ and~$r$
	such that $\allp$,
	$r\in\numN$;
	\begin{alignat*}2
		\frac12			&<\alpha\le1 
		&\quad &\text{for $p=1$},\\
		1-\frac1{2p}	&<\alpha<\frac32-\frac1{2p} 
		&\quad &\text{for $1<p<\infty$},\\
		1				&\le\alpha<\frac32
		&\quad &\text{for $p=\infty$}.
	\end{alignat*}
	Let $f\in\Lp$.
	The following inequality holds true
	\begin{displaymath}
		\norm{\Dl{\arr t r}r{f,x}}
		\le\frac C{\prod_{j=1}^r\Copar{t_j}{}}\norm f,
	\end{displaymath}
	where constant~$C$
	does not depend on~$f$ and~$t_j$ $(j=1,2,\dots,r)$.
\end{cor}

\begin{proof}
	By applying induction with respect to~$r$,
	it is not difficult to see
	(for the analoguous property of the ordinary difference
		for $y_1=y_2=\dots=y_r$ see, e.g.,
		\cite[f.~102]{timan:approximation}
	)
	that the operator $\Dl{\arr t r}r{f,x}$
	can be written in the following form
	\begin{displaymath}
		\Dl{\arr t r}r{f,x}
		=\sum_{k=1}^r(-1)^{k-1}
			\sum_{i_1<\dots<i_k}\T{\arr{\cos y_i}k}k{f,x}
				+(-1)^r f(x),
	\end{displaymath}
	i.e.\ as a linear combination of powers
	$\hatT{\arr{t_i}k}k{f,x}$
	$(i_1<i_2<\dots<i_k;\allowbreak k=0,1,\dots,r)$
	of the appropriate generalised shift operator.
	Now, Corollary~\ref{cr:bound-tau}
	is proved by applying $r$~times Lemma~\ref{lm:bound-tau}.
\end{proof}

\section{Statement of results}

\begin{thm}\label{th:EsubH-tau}
	Let be given numbers~$p$, $\alpha$, $r$ and~$\lambda$
	such that
	$\allp$,
	$r\in\numN$;
	\begin{alignat*}2
		1-\frac1{2p}	&<\alpha<\frac32-\frac1{2p}
			&\quad &\text{for $1\le p<\infty$},\\
		1				&\le\alpha<\frac32
			&\quad &\text{for $p=\infty$}
	\end{alignat*}
	and $0<\lambda<2r$.
	Let $f\in\Lp$.
	If
	\begin{displaymath}
		\E\le Mn^{-\lambda},
	\end{displaymath}
	then
	\begin{displaymath}
		\w\le CM\delta^\lambda,
	\end{displaymath}
	where constant~$C$
	does not depend on~$f$, $M$ and~$\delta$.
\end{thm}

\begin{proof}
	Let $P_n(x)$ be an algebraical polynomial
	of degree not greater than $n-1$
	such that
	\begin{displaymath}
		\norm{f-P_n}=\E \quad(n=1,2,\ldots).
	\end{displaymath}
	We define algebraical polynomials $Q_k(x)$ by
	\begin{displaymath}
		Q_k(x)=P_{2^k}(x)-P_{2^{k-1}}(x)
		\quad (k=1,2,\ldots)
	\end{displaymath}
	and $Q_0(x)=P_1(x)$.
	Since for $k\ge1$
	\begin{multline*}
		\norm{Q_k}
		=\norm{P_2^k-P_{2^{k-1}}}
		\le\norm{P_{2^k}-f}+\norm{f-P_{2^{k-1}}}\\
		=\Epar{2^k}f+\Epar{2^{k-1}}f,
	\end{multline*}
	then by the conditions of the theorem we have
	\begin{equation}\label{eq:Qk-tau}
		\norm{Q_k}\le\Cn M2^{-k\lambda}.
	\end{equation}
	
	Taking into consideration property~\ref{it:properties-tau-4}
	in Lemma~\ref{lm:properties-tau}
	of the operator~$\tau_y$,
	without lost of generality
	we may suppose that $t_s\ne0$ $(s=1,2,\dots,r)$.
	For $0<|t_s|\le\delta$ $(s=1,2,\dots,r)$
	we estimate
	\begin{displaymath}
		I=\norm{\Dl{\arr t r}r{f,x}}.
	\end{displaymath}
	For every positive integer~$N$,
	taking into account property~\ref{it:properties-tau-1}
	in Lemma~\ref{lm:properties-tau}
	and the fact that linearity of the operator $\T t{}{f,x}$
	implies linearity of $\T{\arr t r}r{f,x}$,
	and, in turn,
	linearity of $\Dl{\arr t r}r{f,x}$;
	we get
	\begin{displaymath}
		I\le\norm{\Dl{\arr t r}r{f-P_{2^N},x}}
			+\norm{\Dl{\arr t r}r{P_{2^N},x}}.
	\end{displaymath}
	Since
	\begin{displaymath}
		P_{2^N}(x)=\sum_{k=0}^N Q_k(x),
	\end{displaymath}
	we have
	\begin{multline*}
		I\le\norm{\Dl{\arr t r}r{f-P_{2^N},x}}
			+\sum_{k=0}^N\norm{\Dl{\arr t r}r{Q_k,x}}\\
		=J+\sum_{k=1}^N I_k.
	\end{multline*}
	
	Let~$N$ be chosen in such a way that
	\begin{equation}\label{eq:delta-tau}
		\frac\pi{2^N}<\delta\le\frac\pi{2^{N-1}}.
	\end{equation}
	We prove the following inequalities
	\boxconst{cn:J-tau}
	\begin{equation}\label{eq:J-tau}
		J\le\boxedconst M\delta^\lambda
	\end{equation}
	and
	\begin{equation}\label{eq:Ik-tau}
		I_k\le\Cn M2^{-k\lambda},
	\end{equation}
	where constants~$\boxedconst$ and~$\lastC$
	do not depend on~$f$, $M$, $\delta$ and~$k$.
	
	First we consider~$J$.
	By Corollary~\ref{cr:bound-tau},
	taking into account that $|t_1|\le\delta$,
	we have
	\begin{multline*}
		\norm{\Dl{\arr t r}r{f-P_{2^N},x}}
		\le\frac\Cn{\prod_{j=1}^r\Copar{t_j}{}}\norm{f-P_{2^N}}\\
		=\Cn\Epar{2^N}f
	\end{multline*}
	Therefrom, the condition of the theorem
	and inequality~\eqref{eq:delta-tau}
	yield
	\begin{displaymath}
		\norm{\Dl{\arr tr}r{f-P_{2^N},x}}
		\le\Cn M2^{-N\lambda}
		\le\Cn M\delta^\lambda,
	\end{displaymath}
	which proves inequality~\eqref{eq:J-tau}.
	
	Now we prove inequality~\eqref{eq:Ik-tau}.
	Note that, taking into consideration Corollary~\ref{cr:bound-tau},
	we have
	\begin{displaymath}
		\norm{\Dl{\arr t r}r{Q_k,x}}
		\le\frac{\Cn}{\prod_{j=1}^r\Copar{t_j}{}}\norm{Q_k}.
	\end{displaymath}
	Hence,
	\begin{displaymath}
		I_k
		\le\frac{\Cn}{\prod_{j=1}^r
			\Copar{t_j}{}}M2^{-k\lambda},
	\end{displaymath}
	which proves inequality~\eqref{eq:Ik-tau}.
	
	Inequalities~\eqref{eq:J-tau}, \eqref{eq:Ik-tau}
	and~\eqref{eq:delta-tau}
	yield
	\begin{displaymath}
		I\le\Cn M\prn{\delta^\lambda
			+\sum_{k=1}^N 2^{-k\lambda}}
		\le\Cn M(
			\delta^\lambda
			+2^{-N\lambda}
		)\\
		\le\Cn M\delta^\lambda.
	\end{displaymath}
	
	Theorem~\ref{th:EsubH-tau} is proved.
\end{proof}

\begin{thm}\label{th:HsubE-tau}
	Let be given numbers~$p$, $\alpha$, $r$ and~$\lambda$
	such that $\allp$, $\lambda>0$, $r\in\numN$;
	\begin{alignat*}2
		1-\frac1{2p}	&<\alpha<\frac32-\frac1{2p}
			&\quad \text{for $1\le p<\infty$},\\
		1				&\le\alpha<\frac32 
			&\quad \text{for $p=\infty$}
	\end{alignat*}
	Let $f\in\Lp$.
	If
	\begin{displaymath}
		\w\le M\delta^\lambda,
	\end{displaymath}
	then
	\begin{displaymath}
		\E\le CMn^{-\lambda},
	\end{displaymath}
	where constant~$C$
	does not depend on~$f$, $M$ and~$n$.
\end{thm}

\begin{proof}
	Let $\delta=\frac1n$.
	Then,
	taking into account Theorem~\ref{th:jackson},
	we obtain
	\begin{displaymath}
		\E\le\frac1{\Cn}\wpar{f,\frac1n}
		\le CMn^{-\lambda}.
	\end{displaymath}
	
	Theorem~\ref{th:HsubE-tau} is proved.
\end{proof}

\begin{thm}\label{th:coincidence-tau}
	Let be given numbers~$p$, $\alpha$, $r$ and~$\lambda$
	such that
	$\allp$, $r\in\numN$;
	\begin{alignat*}2
		1-\frac1{2p}	&<\alpha<\frac32-\frac1{2p}
			&\quad &\text{for $1\le p<\infty$},\\
		1				&\le\alpha<\frac32 
			&\quad &\text{for $p=\infty$}
	\end{alignat*}
	Then for $0<\lambda<2r$
	the classes of functions $H(p,\alpha,r,\lambda)$
	coincide between themselves
	for different values of~$r$,
	and they coincide with the class $E(p,\alpha,\lambda)$.
\end{thm}

\begin{proof}
	Note that, under the condition of the theorem,
	Theorem~\ref{th:HsubE-tau} implies the inclusion
	\begin{displaymath}
		H(p,\alpha,r,\lambda)\subseteq E(p,\alpha,\lambda),
	\end{displaymath}
	while Theorem~\ref{th:EsubH-tau} implies the converse inclusion
	\begin{displaymath}
		E(p,\alpha,\lambda)\subseteq H(p,\alpha,r,\lambda).
	\end{displaymath}
	Hence we conclude
	that the assertion of Theorem~\ref{th:coincidence-tau}
	is implied by Theorems~\ref{th:HsubE-tau}
	and~\ref{th:EsubH-tau}.
\end{proof}

Note that analogues of Theorems~\ref{th:HsubE-tau},
\ref{th:EsubH-tau} and~\ref{th:coincidence-tau}
for another generalised shift operator
were proved in~\cite{berisha:math-98}
and, in more general forms,
in~\cite{p-berisha:east-98,b-berisha:anlysis-12}.

\bibliographystyle{plain}
\bibliography{maths}

\end{document}